\documentclass[psamsfonts, reqno, 11pt, letterpaper]{amsart}
\usepackage{amsfonts,amsmath, amsthm, amssymb, latexsym, epsfig}
\usepackage[all]{xy}
\usepackage{xspace}
\usepackage{comment}
\usepackage{setspace}
\usepackage{enumerate}
\usepackage{xcolor}
\usepackage[mathscr]{eucal}
\RequirePackage{pdfsync}

\usepackage{bbold}

	\advance \topmargin by -\headheight
	\advance \topmargin by -\headsep

	\evensidemargin \oddsidemargin
	\newcommand{\ftn}[3]{ #1 : #2 \longrightarrow #3 }
	\newcommand{\ftni}[3]{ #1 : #2 \overset{\cong}{\longrightarrow}#3 }
	\newcommand{\setof}[2]{\left\{ #1 \: : \: #2 \right\}}

	\newcommand{\coker}{\operatorname{coker}}
	
	\newcommand{\multialg}[1]{\mathcal{M}(#1)}
	\newcommand{\corona}[1]{\mathcal{Q}(#1)}
	\newcommand{\Z}{\mathbb{Z}}
	\newcommand{\C}{\mathbb{C}}
	\newcommand{\Q}{\mathbb{Q}}
	
	\newcommand{\N}{\mathbb{N}}
	
	\newcommand{\K}{\mathbb{K}}
	\newcommand{\id}{\operatorname{id}}

	\newcommand{\ksix}{K_{\mathsf{six}}}
	\newcommand{\Ksix}{K_{\mathsf{six}}^{+}}
	\newcommand{\Kscale}{K_{\mathsf{six}}^{+, \Sigma }}
	\newcommand{\wsc}[1]{{#1}^{\Sigma}}
\newcommand{\upos}[2]{[#1]_{#2}}
\newcommand{\class}{\mathcal C}
\newcommand{\I}{\mathfrak{I}}
\newcommand{\Ia}{\mathfrak{J}}
\newcommand{\A}{\mathfrak{A}}
\newcommand{\B}{\mathfrak{B}}
\newcommand{\tKsix}{\widetilde{K}_\textnormal{six}^{+,\Sigma}}
\newcommand{\oo}{\mathbf{1}\mathbf{1}}
\newcommand{\oi}{\mathbf{1}\mathbf{\infty}}
\newcommand{\io}{\mathbf{\infty}\mathbf{1}}
\newcommand{\ii}{\mathbf{\infty}\mathbf{\infty}}

	\newcommand{\D}{\mathfrak{D}}
	\newcommand{\ppp}{\mathsf{p}_{E_3}}

	\numberwithin{equation}{section}

	\theoremstyle{plain}
	\newtheorem{lemma}[equation]{Lemma}
	\newtheorem{theor}[equation]{Theorem}
	\newtheorem{propo}[equation]{Proposition}
	\newtheorem{corol}[equation]{Corollary}
	\newtheorem{conje}[equation]{Conjecture}

	\theoremstyle{definition}
	\newtheorem{defin}[equation]{Definition}
	
	\newtheorem{remar}[equation]{Remark}
	
	\newtheorem{examp}[equation]{Example}

	\numberwithin{figure}{section}

\newenvironment{proofof}[1]{{\noindent\emph{Proof of #1.}}\\}{\hfill$\square$}

	\newcommand{\step}[2]{\noindent{\sc Step #1: #2}\\}
	\newcommand{\mystep}[1]{{\sc Step #1}}

\begin{document}
	\title{The extension problem for graph $C^*$-algebras}

	\author{S{\o}ren Eilers}
        \address{Department of Mathematical Sciences \\
        University of Copenhagen\\
        Universitetsparken~5 \\
        DK-2100 Copenhagen, Denmark}
        \email{eilers@math.ku.dk }
	
		\author{James Gabe}
        \address{School of Mathematics and Statistics\\
University of Glasgow\\
University Place\\
Glasgow G12 8QQ\\
United Kingdom}\curraddr{School of Mathematics and Applied Statistics\\
University of Wollongong\\NSW 2522, Australia}
        \email{jamiegabe123@hotmail.com }
	
	\author{Takeshi Katsura}
        \address{Department of Mathematics, Faculty of Science and Technology, Keio
University, 3-14-1 Hiyoshi, Kouhoku-ku, Yokohama, Japan, 223-8522.}
  \email{katsura@math.keio.ac.jp}
	
	\author{Efren Ruiz}
        \address{Department of Mathematics\\University of Hawaii,
Hilo\\200 W. Kawili St.\\
Hilo, Hawaii\\
96720-4091 USA}
        \email{ruize@hawaii.edu}
        
         \author{Mark Tomforde}
        \address{Department of Mathematics\\University of Houston\\
Houston, Texas\\
77204-3008, USA}
        \email{tomforde@math.uh.edu}
        \date{\today}


	\begin{abstract}
	 We give a complete $K$-theoretical description of when an essential extension of two simple graph $C^{*}$-algebras is again a graph $C^{*}$-algebra.   
	\end{abstract}

        \maketitle

\section{Introduction}

Whenever a class $\class$ of $C^*$-algebras is closed under passing to ideals and quotients, one may ask whether or not  in any extension
\begin{equation}
\xymatrix{
{\mathfrak e}:&0\ar[r]&\I\ar[r]^-{\iota}&\A\ar[r]^-{\pi}&\A/\I\ar[r]&0}\label{ses}
\end{equation}
with $\I,\A/\I\in\class$, it is automatic  that  $\A\in\class$. The class of AF algebras (\cite[Chapter~9]{ege:dca}), the class of Type I $C^*$-algebras (\cite[6.2.6]{gkp:cag}), the class of nuclear $C^*$-algebras (cf.\ \cite[Exercise~3.8.1]{npbno:cfa}), and the class of purely infinite $C^*$-algebras (\cite[Theorem~4.19]{ekmr:npic}) %and -- by design -- the bootstrap class in $\kk$-theory (\cite{jrcs:ktuctkgk}) 
all have this property, but there are also many important classes when such a permanence result fails in obvious ways.  In this case, one may  usefully ask instead whether there is a $K$-theoretical 
description of when one may conclude from membership of $\class$ at the extremes to membership of $\class$ in the middle. In most, if not all, of the instances when such results are known, the $K$-theoretical data used comes from the six term exact sequence
\begin{equation}\vcenter{\xymatrix{
K_0(\I)\ar[r]^-{\iota_0}&K_0(\A)\ar[r]^-{\pi_0}&K_0(\A/\I)\ar[d]^-{\partial_0}\\
K_1(\A/\I)\ar[u]^-{\partial_1}&K_1(\A)\ar[l]^-{\pi_1}&K_1(\I)\ar[l]^-{\iota_1}}}\label{eq:ksix}
\end{equation}
as summarized in Table \ref{samples}.

In Theorem \ref{thm:extensions} of this paper, we provide a result of this nature for the class of graph $C^*$-algebras under the assumption that $\I$ and $\A/\I$ are simple, and $\I$ is not a complemented ideal of $\A$.  Reflecting complications arising from the fact that simple graph $C^*$-algebras may be AF as well as purely infinite, the $K$-theoretical obstructions (which we show by example are all necessary) are rather more complicated than in the cases previously known, with the possible exception of \cite{mdtal:ecrrzc} (which we have not listed in Table \ref{samples} because the class to which it applies is too complicated to describe in the table). We establish our results by a two-step procedure of invoking classification results by $K$-theory paired with range results, thereby proving that $\A$ is a graph $C^*$-algebra by first constructing a graph $E$ whose $C^*$-algebra has the necessary $K$-theoretic data, and then applying classification results to conclude $\A \cong C^*(E)$.

In the case where $\I$ is a complemented ideal, $\A$ will always be a graph $C^*$-algebra since this class is closed under direct sums.
Note also that strictly speaking, the class of graph $C^*$-algebras is not closed under passing to ideals and quotients, as only so-called gauge-invariant ideals respect the structure. Hence we work only with graph $C^*$-algebras with finitely many ideals, where all ideals are gauge-invariant. Importantly for our approach, a graph $C^*$-algebra with finitely many ideals automatically has real rank zero.

\begin{table}
\begin{center}
\begin{tabular}{|l|c|c|}\hline
Class&Obstruction&Reference\\\hline\hline
Real rank zero&$\partial_0=0$&\cite{lgbgkp:crrz}\\\hline
Stable rank one&$\partial_1=0$&\cite{hlmr:eilca}\\\hline
Stably finite&$\operatorname{Im}\partial_1\cap K_0(\I)_+=0$&\cite{jss:eceia}
\\\hline
AT algebras of real rank zero&$\partial_*=0$&\cite{hlmr:eilca}\\\hline
Cuntz-Krieger algebras of real rank zero&$\partial_0=0$&\cite{rb:ecka}\\\hline
\end{tabular}
\end{center}
\caption{Sample extension results.}\label{samples}
\end{table}
\begin{table}
\begin{center}
\begin{tabular}{|c||c|c|}
\hline 
Case & $\mathfrak{I}$ & $\mathfrak{A} / \mathfrak{I}$ \\ \hline\hline
$[\mathbf{1}\mathbf{1}]$ & AF & AF \\ 
$[\mathbf{1} \infty]$ & AF & Kirchberg \\
$[\infty \mathbf{1} ]$ & Kirchberg & AF \\
$[\infty \infty]$ & Kirchberg & Kirchberg \\ \hline
\end{tabular}\qquad 
\begin{tabular}{|c||c|c|}\hline
Case & $\mathfrak{A}$ & $\mathfrak{A} / \mathfrak{I}$ \\ \hline\hline
$\upos{-}{0}$ & Nonunital & Nonunital \\ 
$\upos{-}{1}$ & Nonunital & Unital \\
$\upos{-}{2}$ & Unital & Unital \\\hline
\end{tabular}
\end{center}
\caption{Notation for the 12 cases we consider.}\label{cases}
\end{table}

Providing such a characterization has been the ambition of some of the authors for almost a decade, and the path to obtaining a complete result has been unusually indirect. To describe the contribution of the paper at hand, we introduce notation for the 12 different cases we have to address in proving the result.  As already mentioned, the ideal and quotient are either AF or purely infinite by the dichotomy of simple graph $C^*$-algebras, and we denote the four possible cases as indicated in the left part of Table \ref{cases}. By our assumptions, $\I$ does not have a unit. Further, when $\A$ has a unit, so does $\A/\I$, and hence the number of units among the three $C^*$-algebras $\A$, $\I$, and $\A / \I$ uniquely determines which of these three $C^*$-algebras has a unit. We denote the cases  by $\upos{-}{*}$ with the symbol $*\in\{0,1,2\}$ indicating the number of units among the three $C^*$-algebras, as indicated in the right part of Table \ref{cases}. 

 Back in 2011, the four senior authors thought to have completed the case where $\I$ is stable (this is automatic by \cite{semt:cnga} in all but the $\upos{\oo}{*}$ cases) and in fact the first named author announced this at the Kyoto conference ``$C^*$-Algebras and Applications'' at RIMS. We then went on to pursue the remaining $\upos{\oo}{*}$ case, and had succeeded in solving that by 2013 in \cite{setkermt:iagc}. However, before publishing a  proof of our combined result, Gabe in \cite{jg:nnae} exposed that a claim from \cite{gaedk:avbfat} used in the paper \cite{segrer:okfe} was false (see \cite{segrer:cccabfis} for a detailed corrigendum), thereby  rendering incomplete the proof in a single case, namely the one denoted $\upos{\oi}{1}$. The example Gabe provided to show that classification by the standard invariant fails was, however, not a counterexample to our permanence statement, so we tried at several occasions to remedy this situation. Finding a way took us several years: Indeed, the missing piece of the puzzle was provided in a recent paper \cite{jger:} of Gabe and Ruiz, to the effect of providing a complete classification result applicable to the case left open. Joining forces with Gabe, we were finally able to finish the extension result by providing the necessary range result in the paper at hand. In Table \ref{resappl} we list, as pairs, references to the classification results appropriate range results for each of our 12 cases.

\begin{table}
\begin{center}
\begin{tabular}{|c||c|c|c|}\hline
&$\upos{-}{0}$&$\upos{-}{1}$&$\upos{-}{2}$\\\hline\hline
 $[\oo]$&\cite{gae:cilssfa}, \cite{setkermt:iagc}&\cite{gae:cilssfa}, \cite{setkermt:iagc}&\cite{gae:cilssfa}, \cite{setkermt:iagc}\\ \hline     
 $[\oi]$&\cite{segrer:cecc}, \cite{setkmtjw:rking}&\cite{jger:}, Theorem \ref{takeshitheor}&\cite{segrer:scecc}, \cite{setkmtjw:rking}\\  \hline    
  $[\io]$&\cite{segrer:cecc}, \cite{setkmtjw:rking}&\cite{segrer:okfe}, \cite{setkmtjw:rking}&\cite{segrer:scecc}, \cite{setkmtjw:rking}\\    \hline  
   $[\ii]$&\cite{mr:ceccstesk}, \cite{setkmtjw:rking}&\cite{segrer:okfe}, \cite{setkmtjw:rking}&\cite{segr:rccconi}, \cite{setkmtjw:rking}\\  \hline  
   \end{tabular}
   \end{center}
   \caption{References for the classification result (listed first) and range result (listed second) in each of our 12 cases.}\label{resappl}
   \end{table}

We note that  among the results listed in Table \ref{samples},  the one by Bentmann  is closely and subtly related to the result presented here, and it is worth discussing how our work is related to \cite{rb:ecka}. Indeed, it was an important  open question whether a unital extension of  Cuntz-Krieger algebras is itself in the Cuntz-Krieger class (it is necessary to stabilize the ideal for such a statement to be non-trivial), but since Restorff's classification result for Cuntz-Krieger algebras (\cite{gr:cckasi}) applies only when the middle $C^*$-algebra is known to be Cuntz-Krieger, we were only able to argue in such a way in the few cases when an appropriate external classification result was known, cf.\ \cite{sea:dpcke}. For instance, the necessary result for extensions of simple Cuntz-Krieger algebras was provided in \cite{segr:rccconi}.

Bentmann was able to circumvent this issue by elaborating on an idea from \cite{jc:cctmc2} to make contact to a deep result by Kirchberg (\cite{ek:nkmkna}), providing a complete solution of the extension problem in the Cuntz-Krieger case, and leading us to the following conjecture:

\begin{conje}
Let $C^*(E_1)$ and $C^*(E_3)$ be unital graph $C^*$-algebras with finitely many ideals, and consider the unital extension
\[
\xymatrix{0\ar[r]&C^*(E_1)\otimes \K\ar[r]&\mathfrak{X}\ar[r]&C^*(E_3)\ar[r]&0.}
\]
Then $\mathfrak X$ is a graph $C^*$-algebra if and only if
$\partial_0:K_0(C^*(E_3))\to K_1(C^*(E_1))$ vanishes.
\end{conje}

The forward implication is known to hold, and Bentmann's result establishes the converse when both $E_1$ and $E_3$ are finite graphs with no sinks and no sources. By a slight refinement of our main result, we are able to remove both the finiteness and the ``no sinks and no sources" hypotheses in the very special case when both $C^*(E_1)$ and $C^*(E_3)$ are simple (see Corollary \ref{bentmannplus}). At this stage we do not have a workable conjecture for the general extension problem for graph $C^*$-algebras with finitely many ideals.

The paper is organized as follows: We summarize notation and provide a few preliminary lemmas in Section \ref{prelims}. In Section \ref{remain} we solve the case hitherto left open.  In Section \ref{main} we state our extension result and succinctly explain how existing results may be combined to establish the extension result in the remaining 11 cases mentioned above.

\section{Preliminaries}\label{prelims}

We follow the notation and definition for graph $C^*$-algebras in \cite{flr:graph}.  In particular, our arrows are drawn in the direction for which sinks and infinite emitters are singular.

\begin{defin}
Let $\mathfrak{A}$ be a $C^{*}$-algebra.  An element $a \in \mathfrak{A}$ is \emph{norm-full} if $a$ is not contained in any closed two-sided ideal of $\mathfrak{A}$.  An extension $\mathfrak{e} : 0 \to \mathfrak{I} \to \mathfrak{A} \to \mathfrak{A}/\I \to 0$ is \emph{full} if for every nonzero $x \in \A/\I$, $\tau_{ \mathfrak{e} } ( x )$ is norm-full in $\corona{ \mathfrak{I} }$, where $\multialg{ \mathfrak{I} }$ is the multiplier algebra of $\mathfrak{I}$, $\corona{\I} = \multialg{ \mathfrak{I} } / \mathfrak{I} $, and $\ftn{ \tau_{ \mathfrak{e} } }{{\A/\I} }{ \corona{ \mathfrak{I} } }$ is the Busby invariant associated to $\mathfrak{e}$.
\end{defin}

\begin{lemma}\label{newlem}
Let $\mathfrak{A}$ be a $C^{*}$-algebra with exactly one proper nontrivial ideal $\mathfrak{I}$ satisfying the following four conditions:
\begin{enumerate}[(1)]
\item $\I$ is AF,
\item $\A/\I$ is purely infinite,
\item $\A$ is separable, and
\item $K_0( \mathfrak{A} )_{+} = K_{0} ( \mathfrak{A} )$.
\end{enumerate}
Then $\I$ is stable, and $\frak e$ is a full extension.
\end{lemma}

\begin{proof}
Suppose (1)--(4) hold. Since the embedding from $\A$ to $\A \otimes \K$ induces an order isomorphism from $K_0(\A)$ to $K_0(\A\otimes \K)$, we have $K_0 ( \A \otimes \K )_+ = K_0 ( \A \otimes \K)$.  Note that $\A\otimes \K$ has exactly one proper nontrivial ideal $\I \otimes \K$ (which is an AF algebra) and $\A \otimes \K / ( \I \otimes \K ) \cong (\A/\I) \otimes \K$ is a purely infinite simple $C^*$-algebra.  Also, since $\partial_0$ must vanish, $\A$ has real rank zero by \cite[Theorem~3.19]{lgbgkp:crrz}, and by \cite[Proposition~3.15]{lgbgkp:crrz}, so does $\A \otimes \K$.  Therefore, $\A \otimes \K$ has stable weak cancellation by \cite[Lemma~3.15]{segrer:okfe}, and by \cite[Corollary~3.22]{segrer:okfe}, 
\[
\mathfrak{e}^s: 0 \to \I \otimes \K \to \A \otimes \K \to \A/\I \otimes \K \to 0
\]
is a full extension.  

Let $\{ e_n \}$ be an approximate identity for $\A$ consisting of projections, and let $\{ e_{i,j}\}$ be a system of matrix units for $\K$.  Without loss of generality, we may assume that $e_n$ is not an element of $\I$ for all $n$.  Since 
\[
\I \cong \I \otimes e_{1,1} = \overline{ \bigcup_n ( e_n \otimes e_{1,1} ) (\I \otimes \K) ( e_n \otimes e_{1,1} )}
\]
 with $ ( e_n \otimes e_{1,1} ) (\I \otimes \K) ( e_n \otimes e_{1,1} ) \subseteq  ( e_{n+1} \otimes e_{1,1} ) (\I \otimes \K) ( e_{n+1} \otimes e_{1,1} )$, by \cite[Corollary~4.1]{jhmr:sc}, $\I$ is stable if $( e_n \otimes e_{1,1} ) (\I \otimes \K) ( e_n \otimes e_{1,1} )$ is stable for all $n$.
 
Let $n \in \N$.  Let $\tau_{ \mathfrak{e} }^s \colon \mathfrak A \otimes \K \to \mathcal{Q} ( \I \otimes \K )$ be the Busby map and let $\sigma_{\mathfrak{e}}^s \colon \A \otimes \K \to \mathcal{M} ( \I \otimes \K)$ be the injective $*$-homomorphism such that the diagram
\[
\xymatrix{
0 \ar[r] & \I \otimes \K \ar@{=}[d] \ar[r] & \A \otimes \K \ar[r]^-{\pi} \ar[d]^{\sigma_{\mathfrak{e}}^s} & (\A \otimes \K ) / ( \I \otimes \K ) \ar[r] \ar[d]^{\tau_{\mathfrak{e}}^s} & 0 \\
0 \ar[r] & \I \otimes \K \ar[r] & \mathcal{M} ( \I \otimes \K ) \ar[r] & \mathcal{Q} ( \I \otimes \K ) \ar[r] & 0 
}
\]
commutes.  By the commutativity of the diagram, 
\[
( e_n \otimes e_{1,1} ) (\I \otimes \K) ( e_n \otimes e_{1,1} ) = \sigma_{\mathfrak{e} }^s ( ( e_n \otimes e_{1,1} ) ) (\I \otimes \K) \sigma_{\mathfrak{e} } ( ( e_n \otimes e_{1,1} ) ).
\]
Since $\mathfrak{e}^s$ is a full extension and $e_n \otimes e_{1,1}$ is not an element of $\I \otimes \K$, $\tau_{\mathfrak{e}}^s ( \pi( e_n \otimes e_{1,1} ) )$ is a full projection in $\mathcal{Q} ( \I \otimes \K )$.  By \cite[Lemma~3.3]{segrer:okfe} and the commutativity of the diagram, $\sigma_{\mathfrak{e} }^s ( ( e_n \otimes e_{1,1} ) )$ is a norm-full projection in $\mathcal{M} ( \I \otimes \K )$.  Since $\I \otimes \K$ is an AF algebra, every norm-full projection in $\mathcal{M}( \I \otimes \K)$ is Murray-von Neumann equivalent to $1_{ \mathcal{M}( \I \otimes \K)}$.  Consequently, 
\[
( e_n \otimes e_{1,1} ) (\I \otimes \K) ( e_n \otimes e_{1,1} ) = \sigma_{\mathfrak{e} }^s ( ( e_n \otimes e_{1,1} ) ) (\I \otimes \K) \sigma_{\mathfrak{e} } ( ( e_n \otimes e_{1,1} ) ) \cong \I \otimes \K
\]
which implies that $( e_n \otimes e_{1,1} ) (\I \otimes \K) ( e_n \otimes e_{1,1} )$ is a stable $C^*$-algebra.

By \cite[Theorem~3.22]{segrer:okfe}, $\mathfrak{e}$ is a full extension.
\end{proof}

\begin{defin}
If $\A$ is a $C^*$-algebra, the \emph{scale} of $K_0(\A)$ is the subset
$$ \Sigma \mathfrak{A} = \setof{ x \in K_{0} ( \mathfrak{A} ) }{ \text{$x = [ p ]$ for some projection $p$ in $\mathfrak{A}$}}.$$
If $\A$ and $\B$ are $C^*$-algebras, a homomorphism $\ftn{ \alpha }{ K_{0} ( \mathfrak{A} ) }{ K_{0} ( \mathfrak{B} ) }$ is \emph{contractive} if $\alpha(\Sigma \A)\subseteq \Sigma B$.  We also define an invertible  homomorphism $\ftn{ \alpha }{ K_{0} ( \mathfrak{A} ) }{ K_{0} ( \mathfrak{B} ) }$ to be \emph{scale-preserving} if both $\alpha$ and $\alpha^{-1}$ are contractive, and in the event that $\A$ and $\B$ are both unital we also require $\alpha ([1_\A]) = [1_\B]$.  (If either of $\A$ or $\B$ is nonunital, this last condition imposes no requirement.)

Some comments on this definition are in order:  
When $\A$ is unital and AF, the element $[1_\A]$ is the maximal element of $\Sigma \A$ and in fact, $\Sigma \A = \{ x \in K_0(\A) : 0 \leq x \leq [1_\A] \}$.  Consequently, if $\A$ and $\B$ are both unital AF algebras, an order-preserving bijection $\alpha : K_0(\A) \to K_0(\B)$ is scale-preserving if and only if $\alpha$ and $\alpha^{-1}$ are contractive.  (Thus, the condition $\alpha ([1_\A]) = [1_\B]$ in the definition of scale-preserving is superfluous in the AF case.)  On the other hand, if $A$ is purely infinite and simple, then $\Sigma \mathfrak{A} = K_0(\A)$.  Consequently, if $\A$ and $\B$ are both purely infinite and simple, any homomorphism between their $K_0$-groups is automatically contractive, so that a bijection $\alpha : K_0(\A) \to K_0(\B)$ is scale-preserving if and only if $\alpha ([1_\A]) = [1_\B]$ in the event that $\A$ and $\B$ are unital, and always so in the event at least one of $\A$ and $\B$ is nonunital.  (Thus, the condition that $\alpha$ and $\alpha^{-1}$ be contractive in the definition of scale-preserving is superfluous in the purely infinite case.)
\end{defin}

When $\A$ is AF, the scale $\Sigma \A$ is a subset of $K_0(\A)$ with the following  properties:
\begin{itemize}
\item generating ($\forall x\in K_0(\A)_+,\ \exists y_1, \ldots, y_n  \in \Sigma\A:x=y_1+\cdots+y_n$);
\item hereditary ($\forall x\in K_0(\A)_+, \ \forall y\in \Sigma \A: x\leq y\Longrightarrow x\in\Sigma \A$); and
\item upward directed ($\forall y_1,y_2\in\Sigma\A, \ \exists z\in \Sigma\A:y_1 \leq z \text{ and } y_2 \leq z$).
\end{itemize}

We recall that by the result of Effros, Handelman and Shen (\cite{ehs}), the range of ordered $K_0$-groups associated to (separable) AF algebras are exactly the (countable) \emph{dimension groups}: ordered groups with the Riesz interpolation property that are also unperforated. 
As explained in \cite{ege:dca} it is possible to amend this result by earlier work of Elliott  (\cite{gae:cilssfa})  to see that any  dimension group $(G,G_+)$ with a subset $\Sigma\subseteq G_+$ that is generating, hereditary, and upward directed arises as the scaled $K_0$-group of an AF algebra.

Let $\mathfrak{A}$ be a $C^{*}$-algebra and $\mathfrak{I}$ be an ideal of $\mathfrak{A}$.  Let $\ksix ( \mathfrak{A} , \mathfrak{I} )$ denote the six-term exact sequence \eqref{eq:ksix} in $K$-theory
induced by the extension \eqref{ses}.
$\Ksix ( \mathfrak{A} , \mathfrak{I} )$ will denote the same sequence as in (\ref{eq:ksix}), where the three $K_{0}$-groups are considered as pre-ordered groups.  We let $\Kscale ( \mathfrak{A} , \mathfrak{I} )$ denote the same sequence as in (\ref{eq:ksix}) but $K_{0} ( \mathfrak{A} )$ is considered as a pre-ordered group, and $K_{0} ( \mathfrak{A} / \mathfrak{I} )$ and $K_{0} ( \mathfrak{I} )$ are considered as scaled pre-ordered groups.

For a $C^{*}$-algebra $\mathfrak{B}$ with ideal $\Ia$, a homomorphism $\ftn{ ( \alpha_{\bullet},\beta_\bullet) }{ \ksix ( \mathfrak{A} , \mathfrak{I} ) }{ \ksix ( \mathfrak{B} , \Ia ) }$ consists of six group homomorphisms $\alpha_1,\alpha_2,\alpha_3,\beta_1,\beta_2,\beta_3$, making the following diagram commute:
\begin{eqnarray}\label{isodef}\vcenter{
\xymatrix{ K_0(\I) \ar[rr]^-{\iota_0} \ar[rd]_-{\alpha_1} & & 
K_0(\A) \ar[rr]^-{\pi_0} \ar[d]_-{\alpha_2} & & 
K_0(\A/\I) \ar[ddd]^-{\partial_0} \ar[dl]^-{\alpha_3} \\ 
& {K_0(\Ia)}\ar[r]^-{\iota_0} & {K_0(\B)}\ar[r]^-{\pi_0} 
& {K_0(\B/\Ia)}\ar[d]^-{\partial_0} & \\ 
& {K_0(\B/\Ia)}\ar[u]^-{\partial_1} & {K_1(\B)} \ar[l]_-{\pi_1} 
& {K_1(\Ia)}\ar[l]_-{\iota_1} & \\ 
K_1(\A/\I) \ar[uuu]^-{\partial_1} \ar[ru]^-{\beta_3} & & 
K_1(\A) \ar[ll]_-{\pi_1} \ar[u]^-{\beta_2} & &
K_1(\I) \ar[ll]_-{\iota_1} \ar[ul]_-{\beta_1}}}
\end{eqnarray}
  An isomorphism $\ftni{ ( \alpha_{\bullet},\beta_\bullet ) }{ \ksix ( \mathfrak{A} , \mathfrak{I} ) }{ \ksix ( \mathfrak{B} , \Ia ) }$ is defined in the obvious way.  

Homomorphisms and isomorphisms from $\Ksix ( \mathfrak{A} , \mathfrak{I} )$ to $\Ksix ( \mathfrak{B} , \Ia )$ and from $\Kscale ( \mathfrak{A} , \mathfrak{I} )$ to $\Kscale ( \mathfrak{B} , \Ia )$ are defined in a similar way with the requirement that the additional structure is preserved. We recall that a scale-preserving isomorphism must preserve the classes of the units also in the purely infinite unital case, when this would not follow from the map only be required to send scales to scales.

\section{The outstanding $\upos{\io}{1}$ case}\label{remain} 

In this section we prove the $K$-theoretical existence result necessary to resolve the extension problem in the case where $\I$ is AF, $\A/\I$ is unital and purely infinite, and $\A$ is non-unital. To do so, we carefully describe the invariant used in \cite{jger:} and elaborate on methods from \cite{setkmtjw:rking} to make contact to a result by the third and fifth named authors with Sims (\cite{tkasmk:ragaelua}) which establishes that certain AF algebras are always graph $C^*$-algebras.

Given a short exact sequence of $C^*$-algebras 
\[
\xymatrix{\frak e:&
0 \ar[r] & \I \ar[r]^{\iota} & \A \ar[r]^{\pi} & \A/\I \ar[r] & 0
}
\]
such that the quotient  $\A/\I$ has a unit $1_{\A/\I}$ 
and this unit lifts to a projection in $\A$, 
we define $\tKsix(\A,\I)$ to be 
\begin{equation} \label{Ksix}
\vcenter{\xymatrix{
0 \ar[r] & \wsc{K}_0(\I) \ar[r]^-{\iota_0} \ar[d]^-{\id_0}& 
\wsc{K}_0(\D) \ar[r]^-{\pi_0} \ar[d]^-{\iota_0}
& \wsc{K}_0(\C 1_{\A/\I}) \ar[d]^-{\iota_0}  \ar[r] & 0\\
& K_0(\I)\ar[r]^-{\iota_0} & K_0(\A) \ar[r]^-{\pi_0} & \wsc{K}_0 (\A/\I) \ar[d]^-{\partial_0} & \\
& K_1 (\A/\I) \ar[u]^-{\partial_1} & K_1 (\A)  \ar[l]^-{\pi_1} & K_1(\I) \ar[l]^-{\iota_1} & 
} }
\end{equation}
where $\D := \pi^{-1}(\C 1_{\A/\I}) \subseteq \A$. The top row is exact since $K_1(\C 1_{\A/\I})$ vanishes and 
since $1_{\A/\I}$ lifts to a projection in $\A$. The latter observation further shows that there is a splitting map for the top row. The notation $\wsc{K}_0$ for $K_0$ helps to distinguish the two groups in the top left corner, and to remind us that at the positions of these four groups, we will  impose conditions of contractiveness and scale-preservation. We note:

\begin{lemma}\label{basicD}
If $\frak e$ is a full extension, then so is
\[
\xymatrix{\widetilde{\frak e}:&
0 \ar[r] & \I \ar[r]^{\iota} & \D \ar[r]^{\pi} & \C 1_{\A/\I} \ar[r] & 0.
}
\]
\end{lemma}
\begin{proof}
Follows from the fact that $\tau_{ \mathfrak{e}}(1_{\mathfrak A /\mathfrak I}) = \tau_{\widetilde{\mathfrak{e}}} (1)$.
\end{proof}

 An isomorphism $\ftni{ (\widetilde{\alpha}_{\bullet}, \alpha_{\bullet},\beta_\bullet ) }{ \tKsix ( \mathfrak{A} , \mathfrak{I} ) }{ \tKsix ( \mathfrak{B} , \Ia ) }$ consists of nine group homomorphisms such that  $(\alpha_{\bullet},\beta_\bullet )$ induces an isomorphism from $\Kscale ( \mathfrak{A} , \mathfrak{I} )$ to $\Kscale ( \mathfrak{B} , \Ia )$ as defined in the previous section, and so that the $\widetilde{\alpha}_{\bullet}$ are all scale-preserving. All maps must commute with the morphisms in \eqref{Ksix}.

We will use the following result by the second and fourth named authors:

\begin{theor}[{\cite[Theorem~B]{jger:}}]\label{birthdaypresent}
With notation and assumptions as above, assume that $\A$ is a $C^*$-algebra with exactly one proper nontrivial ideal $\I$ so that the following hold:
\begin{enumerate}[(i)]
\item $\I$ is stable and AF,
\item $\A/\I$ is a unital Kirchberg algebra in the UCT class, and
\item $\mathfrak e$ is a full extension.
\end{enumerate}
Then $\tKsix(\A,\I)$ is a complete invariant for $\A$ 
in the sense that when $\A'$ is another such  $C^*$-algebra with exactly one proper nontrivial ideal $\I'$ satisfying (i)--(iii), then
\[
\tKsix(\A,\I)\cong \tKsix(\A',\I')\Longleftrightarrow \A\cong \A'.
\]
\end{theor}

The class of invariants $\mathcal{\widetilde{E}}$ we consider may be abstractly characterized as consisting of a
cyclic six-term exact sequence 
\begin{equation} \label{six-term}
\vcenter{\xymatrix{
G_1 \ar[r]^-{\epsilon} & G_2 \ar[r]^-{\gamma} & G_3 \ar[d]^-{\delta_0} \\
F_3 \ar[u]^-{\delta_1} & F_2 \ar[l]_-{\gamma'} & F_1 \ar[l]_-{\epsilon'} 
}}
\end{equation}
of abelian groups, and a short exact sequence 
\begin{equation} \label{short-exact}
\vcenter{\xymatrix{
0 \ar[r] & H_1 \ar[r]^-{\widetilde{\epsilon}} & 
H_2 \ar[r]^-{\widetilde{\gamma}} & H_3 \ar[r] & 0
}}
\end{equation}
of scaled ordered groups together with 
three group homomorphisms $\eta_i\colon H_i \to G_i$ 
such that the diagram 
\begin{equation} \label{scale}
\vcenter{\xymatrix{
H_1 \ar[r]^-{\widetilde{\epsilon}} \ar[d]^-{\eta_1}& 
H_2 \ar[r]^-{\widetilde{\gamma}} \ar[d]^-{\eta_2}
& H_3 \ar[d]^-{\eta_3} \\
G_1 \ar[r]^-{\epsilon} & G_2 \ar[r]^-{\gamma} & G_3 
}}
\end{equation}
commutes. 
We summarize the above information 
into the following diagram: 
\begin{equation} \label{invariant}
\vcenter{\xymatrix{
0 \ar[r] & H_1 \ar[r]^-{\widetilde{\epsilon}} \ar[d]^-{\eta_1}& 
H_2 \ar[r]^-{\widetilde{\gamma}} \ar[d]^-{\eta_2}
& H_3 \ar[d]^-{\eta_3}  \ar[r] & 0\\
& G_1 \ar[r]^-{\epsilon} & G_2 \ar[r]^-{\gamma} & G_3 \ar[d]^-{\delta_0} & \\
& F_3 \ar[u]^-{\delta_1} & F_2 \ar[l]_-{\gamma'} & F_1 \ar[l]_-{\epsilon'} & 
}}
\end{equation}
where we require throughout that $\widetilde{\epsilon}$ and $\widetilde{\gamma}$ are positive and contractive homomorphisms, and that $\eta_1,\eta_2,\eta_3,\epsilon,\gamma$ are positive homomorphisms.

%The isomorphism between such invariants is defined in the obvious way, involving nine group isomorphisms that commute with the homomorphisms of the given diagram. The three isomorphisms $\widetilde{\alpha}_\bullet$ at the top level must further be scale-preserving, the three isomorphisms $\alpha_\bullet$ in the middle must be order isomorphisms with $\alpha_3$ scale-preserving (i.e. preserving the given order unit), whereas no restrictions are imposed on the isomorphisms $\beta_\bullet$ at the lower level. 

\begin{theor}\label{takeshitheor}
Suppose 
 $\mathcal{\widetilde{E}}$ is an invariant so that the following nine conditions hold:
 \begin{enumerate}[(i)]
\item $H_1$ is a simple dimension group with
$\Sigma H_1 = H_1^+$,
\item $
H_2^+ = \widetilde{\epsilon}(H_1^+) \cup \widetilde{\gamma}^{-1}(\{1,2,\ldots\})$ and $ \Sigma H_2 \cap \widetilde{\gamma}^{-1}(\{1\})\not=\emptyset$,
\item  $\Sigma H_2$ is generating, hereditary and upward directed, and does not have a largest element;
\item  $(H_3,(H_3)_+,\Sigma H_3) \cong (\Z,\N\cup\{0\},\{0,1\})$,
\item $\eta_1$ is an isomorphism of ordered groups,
\item $(G_2)_+=G_2$,
\item $G_3$ is finitely generated and $(G_3)_+=G_3$,
\item $F_1=0$, and
\item $F_3$ is a free group with $\operatorname{rank} F_3\leq \operatorname{rank} G_3$.
\end{enumerate}
Then there is a graph $C^*$-algebra $C^*(E)$ with exactly one proper nontrivial ideal $\I$ so that $C^*(E)/\I$ is unital and
$\tKsix(C^*(E),\I)\cong \mathcal{\widetilde{E}}$.
\end{theor}

To prove Theorem \ref{takeshitheor}
we need the following auxiliary results. The first result, 
which is a generalization of \cite[Proposition~4.8]{setkmtjw:rking}, will allow us to glue together two graphs representing $H_1,G_1,F_1$ and $H_3,G_3,F_3$, respectively. 
Note that in \cite[Proposition~4.8]{setkmtjw:rking}, 
it is assumed $n_1 <\infty$, and $x \in \Z^{n_1}$ 
in the proposition below is given by $\mathbf 1$. The version given below also corrects a number of regrettable typos in the statement of \cite[Proposition~4.8]{setkmtjw:rking}.

 We recall the setting from \cite[Proposition~4.3]{setkmtjw:rking}. As in \eqref{six-term} above, we let $\mathcal{E}$  denote an exact sequence of abelian groups
with $F_1$, $F_2$, and $F_3$ free and suppose that there exist column-finite matrices 
$A \in M_{n_1, n_1'} (\Z)$ 
and $B \in M_{n_3, n_3'} (\Z)$
for some $n_1, n_1', n_3, n_3' \in \{0,1,2,\ldots,\infty\}$ 
with isomorphisms 
\begin{align*}
\alpha_1 &\colon \coker A \to G_1, &
\beta_1 &\colon \ker A \to F_1, \\ \nonumber
\alpha_3 &\colon \coker B \to G_3, &
\beta_3 &\colon \ker B \to F_3. 
\end{align*}

We prove in  \cite[Proposition~4.3]{setkmtjw:rking} that there exist a column-finite matrix $Y \in M_{n_1, n_3'} (\Z)$
and isomorphisms 
\begin{align} \label{fourthreea}
\alpha_2 \colon \coker \left(\begin{smallmatrix} A & Y \\ 0 & B \end{smallmatrix} \right) \to G_2, &
& \beta_2 &\colon \ker \left(\begin{smallmatrix} A & Y \\ 0 & B \end{smallmatrix} \right) \to F_2
\end{align}
such that $\alpha_i$ and $\beta_i$ for $i=1,2,3$ 
give an isomorphism (see \eqref{isodef})  from the exact sequence
\begin{equation}\label{fourthreeb}
\vcenter{\xymatrix{ 
\coker A \ar[r]_-{I} 
& \coker {\left( \begin{smallmatrix} A & Y \\ 0 & B \end{smallmatrix} \right)} \ar[r]_-{P} & \coker B\phantom{.} \ar[d]^0 \\
\ker B \ar[u]^{[Y]} &  \ker  {\left( \begin{smallmatrix} A & Y \\ 0 & B \end{smallmatrix} \right)} \ar[l]_-{P'} & \ker A \ar[l]_-{I'}
}}
\end{equation}
to $\mathcal{E}$, where $I,I'$ and $P,P'$ 
are induced by the obvious inclusions or projections.

In the proposition below, the ordering on matrices is given by entrywise ordering, i.e., for matrices $A, B$, we write $A \geq B$ if $A_{i,j} \geq B_{ i,j}$ for all $(i,j)$.

\begin{propo}\label{takeshipropo} 
In the situation described above 
assume that $n_3<\infty$, 
and that $g_2\in G_2$ is given with 
$\alpha_3([\mathbf 1])=\gamma(g_2)$. 
Choose $x \in \Z^{n_1}$ arbitrary 
and define $y \in \Z^{n_1+n_3}$ by 
$y=\left(\begin{smallmatrix}x\\\mathbf 1 \end{smallmatrix}\right)$. 
Suppose that $B$ satisfies the condition that for some $1\leq i,j \leq n_3$ we have
\[
B_{ik}>B_{jk}\qquad 1\leq k \leq n_3'.
\]
Then for a given $Z\in \mathsf{M}_{n_1,n_3'}(\Z)$, 
the matrix $Y\in \mathsf{M}_{n_1,n_3'}(\Z)$ along with $\alpha_2$, $\beta_2$ 
inducing the isomorphism may be chosen 
with the additional properties $Y\geq Z$ and $\alpha_2([y])=g_2$.
\end{propo}

\begin{proof}
Take $Y$, $\alpha_2$, and $\beta_2$ as in \eqref{fourthreea}--\eqref{fourthreeb}.
We are going to define new $Y'$ and $\alpha_2'$ ($\beta_2$ is unchanged) 
which satisfy \eqref{fourthreea}--\eqref{fourthreeb} 
as well as $Y'\geq Z$ and $\alpha'_2([y])=g_2$. 
Set
\[
g_2'=\alpha_2([y])-g_2,
\]
and observe that $\gamma(g_2')=0$ because 
\[
\gamma(\alpha_2([y]))=
\alpha_3(P([y]))=\alpha_3([\mathbf 1])=\gamma(g_2).
\]
Hence there exists $z\in\Z^{n_1}$ such that $\epsilon(\alpha_1([z]))=g_2'$. Choose $Q'\in \mathsf{M}_{n_1,n_3}(\Z)$ such that $z=Q'\mathbf 1$, which is possible because $n_3\geq 1$.
Let $Q''\in \mathsf{M}_{n_1,n_3}(\Z)$ be
\[
(Q'')_{k,\ell}=\begin{cases} 1&\ell=i\\-1&\ell=j\\0&\text{otherwise}
\end{cases}
\]
and note that each row of $Q''B$ is identically
\[
\begin{pmatrix}
B_{i,1}-B_{j,1}&B_{i,2}-B_{j,2}&\cdots& B_{i,n_3'}-B_{j,n_3'}
\end{pmatrix}
\]
which is strictly positive by assumption on $B$. 
Find an integer $c>0$ so that we have
\[
cQ''B\geq Z-Y-Q'B
\]
and set $Q=cQ''+Q'$. 
Since $Q''\mathbf 1 =0$, we have $Q \mathbf 1 =z$.
Set $Y'=Y+QB$. 
Since 
\[
\left( \begin{smallmatrix} I&-Q\\0&I\end{smallmatrix} \right)
 \left( \begin{smallmatrix} A &Y'\\ 0 & B \end{smallmatrix} \right)
=
\left(\begin{smallmatrix}A&Y\\ 0 & B \end{smallmatrix} \right)
\]
%Hence $\pi_2'$ factors through a map 
we can define a map
\[
\alpha_2':\coker \left(\begin{smallmatrix} A & Y' \\ 0 & B \end{smallmatrix} \right) 
\to G_2
\]
 by
\[
\alpha_2'([y])=\alpha_2\left(\left[\left( \begin{smallmatrix} I & -Q \\ 0 & I \end{smallmatrix} \right)y\right]\right).
\]
Now it is straightforward to see that $Y'$, $\alpha_2'$ and $\beta_2$
satisfy  \eqref{fourthreea}--\eqref{fourthreeb}. 
It is also easy to see $Y'\geq Z$, and 
we have $\alpha_2'([y]) = g_2$ because 
\begin{gather*}
\alpha_2'([y]) = \alpha_2\left(\left[\left( \begin{smallmatrix} I& -Q \\ 0 & I \end{smallmatrix} \right)\left( \begin{smallmatrix} x \\ \mathbf 1  \end{smallmatrix} \right)\right]\right)  
=\alpha_2\left(\left[
\left( \begin{smallmatrix} x \\ \mathbf 1  \end{smallmatrix} \right)\right]\right)
- \alpha_2\left( \left[\left(\begin{smallmatrix}Q\mathbf 1\\ \mathbf 0 \end{smallmatrix} \right)\right]\right)
  = \alpha_2([y])-\alpha_2\left(\left[\left(\begin{smallmatrix}z\\\mathbf 1\end{smallmatrix}\right)\right]\right)\\
= \alpha_2([y])-\epsilon(\alpha_1([z])) = \alpha_2([y])-g_2' = g_2.
\end{gather*}
\end{proof}

The second result will allow us to realize $\Sigma H_2$ by finding an AF graph $C^*$-algebra  with the scale adjusted to our needs.

\begin{lemma}\label{AFpart}
Let $H_i$ satisfy (i)--(iv) of Theorem \ref{takeshitheor} and fix $h_2\in  \Sigma H_2  \cap \widetilde{\gamma}^{-1}(\{1\})$.
Then
\[
\Sigma := \{h \in H_1^+ \mid h_2+\widetilde{\epsilon}(h) \in \Sigma H_2\}
\]
is a scale of $H_1$ that is  generating, hereditary, and upward directed, and does not have a largest element, and consequently there is a graph $E_1$ with 
\[
(K_0(C^*(E_1)),K_0(C^*(E_1))_+,\Sigma C^*(E_1))\cong(H_1,(H_1)_+,\Sigma).
\]
\end{lemma}
\begin{proof}
It is obvious that $\Sigma$ is hereditary. For the remaining claims we note that whenever $h_2+\widetilde{\epsilon}(h)\leq k\in \Sigma$,  $\widetilde{\gamma}(h_2)=1=\widetilde{\gamma}(k)$ and hence $k-h_2=\widetilde\epsilon (h')\geq 0$ for some $h'$. We have
\[
k=h_2+\widetilde\epsilon (h')\geq h_2+\widetilde\epsilon (h),
\]
so we conclude that $\widetilde\epsilon (h'-h)\geq 0$ in $H_2$. By (ii) this implies that $h'\geq h\geq 0$ in $H_1$, and $h'\in \Sigma$.

Since there is no largest element in $\Sigma H_2$, we may arrange that $h<h'\in \Sigma$ as above, and hence there is no largest element in $\Sigma$. In particular we have that $h'>0$, and since $H_1$ is a simple dimension group, any positive element is dominated by $nh'$ for some $n$. Using Riesz decomposition, this shows that any positive element is a finite sum of scale elements, and hence $H_1$ is generated by elements from $\Sigma$.

To see that  $\Sigma$ is upward directed, fix $h,\overline{h}$ in $\Sigma$ and take $k$ so that $h_2+\widetilde{\epsilon}(h),h_2+\widetilde{\epsilon}(\overline{h})\leq k$. As above we may choose  $h'$ with $k=h_2+\widetilde\epsilon (h')$ and conclude that $h,\overline{h}\leq h'$.

We conclude that $(H_1,(H_1)_+,\Sigma)$ is a scaled dimension group with no largest scale element, and thus by  \cite{tkasmk:ragaelua} we may choose $E_1$ as desired.
\end{proof}

\begin{proofof}{Theorem \ref{takeshitheor}}
 We choose $h_2$ and define 
$\Sigma \subseteq H_1$ as  in Lemma \ref{AFpart}
so that there exists a graph $E_1 = (E_1^0, E_1^1, r_{E_1}, s_{E_1})$ 
with $C^*(E_1)$ realizing $({H}_1,(H_1)_+,\Sigma)$ as scaled ordered groups. 

Let $g_2 := \eta_2(h_2) \in G_2$ and $g_3 := \gamma(g_2) \in G_3$. 
By \cite[Proposition~3.8]{setkmtjw:rking}, there is a graph 
$E_3 = (E_3^0, E_3^1, r_{E_3}, s_{E_3})$ with finitely many vertices,
such that 
\begin{enumerate}[(1)]
\item every vertex in $E_3$ is the base point of at least two loops,
\item $E_3$ is transitive (so that, in particular, $C^*(E_3)$ is simple and purely infinite), 
\item $(K_0(C^*(E_3)),[1_{C^*(E_3)}]_0) \cong (G_3,g_3)$ 
and $K_1(C^*(E_3)) \cong F_3$, and
\item there exist two vertices $v, w \in E_3^0$ 
such that $(R_{E_3}-I)(w,v')<(R_{E_3}-I)(v,v')$ 
for all $v' \in (E_3^0)_\textnormal{reg}$.
\end{enumerate}

Arguing as in \cite[Proposition~5.5]{setkmtjw:rking}, but applying Proposition \ref{takeshipropo} in place of \cite[Proposition~4.7 and~4.8]{setkmtjw:rking},
we construct a graph $E_2=(E_2^0, E_2^1, r_{E_2}, s_{E_2})$ 
such that 
$\Ksix(C^*(E_2) ,\I)$ is isomorphic to 
the 6-term part $\mathcal{E}$ of the invariant $\mathcal{\widetilde{E}}$, and  
with the further property that the isomorphism $\alpha_2 \colon K_0(C^*(E_2)) \to G_2$ 
sends $[\ppp]$ to $g_2 \in G_2$, where
\[
\ppp=\sum_{v\in E_3^0} p_v.
\] 
As in \cite{setkmtjw:rking}, $E_2^0=E_1^0\sqcup E_3^0$, and $E_2^1$ contains all edges in $E_1^1\sqcup E_3^1$ as well as a number of additional edges chosen carefully to obtain the relevant $K$-theoretical data.

We note that 
there is a natural surjection $\pi \colon C^*(E_2) \to C^*(E_3)$ 
whose kernel is isomorphic to $\I$. 
Thus we may identify $C^\ast(E_3)$ with the quotient $C^*(E_2)/\I$, having
 $\pi$ as the quotient map. 
We see that the projection $\ppp\in C^*(E_2)$ is 
a lift of $1_{C^*(E_3)}$, and
note that 
$\D := \pi^{-1}(\C 1_{C^*(E_3)}) \subseteq C^*(E_2)$ 
coincides with $\I + \C \ppp$.

We see that $C^*(E_1)$ is naturally a subalgebra of $C^*(E_2)$, 
which is a full and hereditary subalgebra of $\I$, and that
\[
\D\cap\{\ppp\}^\perp=C^*(E_1)
\]
under this identification.
We claim that 
\begin{equation}\label{scalekey}
\Sigma\D=\{ x \in K_0 ( \D )_+ \mid x \leq [ \ppp ] + [q] \text{ for some projection $q \in C^*(E_1)$} \}.
\end{equation}
Since $\D\cap\{\ppp\}^\perp=C^*(E_1)$ and since $\Sigma \D$ is hereditary, 
\[
\{ x \in K_0 ( \D )_+ \mid x \leq [ \ppp ] + [q] \text{ for some projection $q \in C^*(E_1)$} \} \subseteq \Sigma \D.
\]
We now show the other set containment.  Let $p$ be a projection in $\D$.  Since $\Sigma \D$ is upward directed, there exists a projection $r$ in $\D$ such that $[p] , [\ppp] \leq [r]$.  Since the unitization of $\D$, which we denote $\D^\dagger$, has stable rank one, there exists a unitary $u \in \D^\dagger$ such that $\ppp \leq u r u^*$.  Hence, $q = u r u^* - \ppp \in \D\cap\{\ppp\}^\perp=C^*(E_1)$.  Consequently,
\[
[p] \leq [r] =  [ \ppp ] + [ r ] - [ \ppp ] = [ \ppp ] + [ u r u^* ] -  [ \ppp ] = [ \ppp ] + [ q ].
\]
This proves the claim.

Next, we construct group homomorphisms $\widetilde{\alpha}_\bullet$ 
for $\bullet=1,2,3$ so that 
\begin{equation}\label{bigdia}
\vcenter{\xymatrix{
0 \ar[r] & 
\wsc{K}_0(\I) \ar[r]^-{\iota_0} \ar[d]^-{\widetilde{\alpha}_1} & 
\wsc{K}_0(\D) \ar[r]^-{\pi_0} \ar[d]^-{\widetilde{\alpha}_2} &
\wsc{K}_0(\C1_{C^*(E_3)}) \ar[d]^-{\widetilde{\alpha}_3}  \ar[r] & 0\\
0 \ar[r] & H_1 \ar[r]^-{\widetilde{\epsilon}} & 
H_2 \ar[r]^-{\widetilde{\gamma}} 
& H_3  \ar[r] & 0
}}
\end{equation}
commutes. 
We also need to show that the maps $\widetilde{\alpha}_\bullet$ 
intertwine $\eta_\bullet$ and are isomorphisms of scaled ordered groups. 

We set $\widetilde{\alpha}_1=\eta_1^{-1}\circ \alpha_1\circ \operatorname{id}_0$ and note that $\widetilde{\alpha}_1$ is an order isomorphism since it is a composition of three order isomorphisms.  Since $\I$ is stable by \cite{semt:cnga}, the scale of $\wsc{K}_0(\I)$ coincides 
with its positive cone. 
Since the same is assumed for ${H}_1$, 
we conclude that $\widetilde{\alpha}_1$ is an isomorphism 
of scaled ordered groups. 
We also define  $\widetilde{\alpha}_3$ by $\widetilde{\alpha}_3([1_{C^*(E_3)}])=\widetilde{\gamma}(h_2)$, and note  that  
both $\widetilde{\alpha}_1$ and $\widetilde{\alpha}_3$ intertwine the appropriate $\eta_\bullet$ by construction.

Next, we define $\widetilde{\alpha}_2$ as the unique group homomorphism 
satisfying 
$\widetilde{\alpha}_2 \circ \iota_0 
= \widetilde{\epsilon} \circ \widetilde{\alpha}_1$ 
and $\widetilde{\alpha}_2([\ppp])=h_2$. 
The map $\tilde{\alpha}_2$ is automatically an order isomorphism as a consequence of the commutativity in (4.7), the fact that $\tilde{\alpha}_1$ and $\tilde{\alpha}_3$ are order isomorphisms, and the way in which the order structures on $H_2$ and $K^\Sigma_0(\D)$ depend on the order structures at the extremes of the diagram. 
Also, $\widetilde{\alpha}_2$ intertwines $\eta_2$ on the image of $\iota_0$ since
\[
\eta_2\circ \widetilde{\alpha}_2\circ \iota_0 =\eta_2\circ \widetilde{\epsilon} \circ \widetilde{\alpha}_1=\epsilon\circ \eta_1\circ \eta_1^{-1}\circ \alpha_1\circ \operatorname{id}_0
=\alpha_2\circ \iota_0
\]
and on $[\ppp]$ by construction. Thus all that remains is to prove that $\widetilde{\alpha}_2(\Sigma \mathfrak D)=\Sigma H_2$.

When $x=[p]\in \Sigma \mathfrak D$,  $[p] \leq [q]+[\ppp]$ for some projection $q$ in $C^*(E_1)$ by \eqref{scalekey}.  Therefore, 
\[
0 \leq \widetilde{\alpha}_2 ([p]) \leq \widetilde{\epsilon}(\widetilde{\alpha}_1([q])) + h_2 \in \Sigma H_2,
\]
because $\widetilde{\alpha}_1([q])\in \Sigma$ by construction.  Since $\Sigma H_2$ is hereditary, $\widetilde{\alpha}_2 ([p])  \in \Sigma H_2$.

For fixed $h\in \Sigma H_2$, take $k\in \Sigma H_2$ so that $h,h_2\leq k$. As in the proof of Lemma \ref{AFpart}, we have $k=h_2+\widetilde{\epsilon}(h')$ with $h'\in \Sigma$, proving that $k=\widetilde{\alpha}_2(\Sigma \mathfrak D)$. Since $\tilde{\alpha}_2 (\Sigma \D))$ is an order-isomorphic image of a hereditary set, it is also hereditary, and we conclude that $h \in \tilde{\alpha}_2 (\Sigma \D)$.

\end{proofof}

\section{The main result}\label{main}

We are now ready to state and prove the main result of the paper. Recall that the (torsion-free) rank of an abelian group $G$ is the dimension of the $\Q$-vector space $G\otimes \Q$.

\begin{theor}\label{thm:extensions}
Let $\mathfrak{A}$ be a $C^{*}$-algebra with exactly one proper nontrivial ideal $\mathfrak{I}$ so that  $\mathfrak{I}$ and $\mathfrak{A} / \mathfrak{I}$ are  graph $C^{*}$-algebras.  Then $\mathfrak{A}$ is a graph $C^{*}$-algebra if and only if the following three conditions hold:
\begin{itemize}
\item[(1)] The exponential map $\ftn{\partial_{0}}{K_{0} ( \mathfrak{A} / \mathfrak{I} ) }{ K_{1} ( \mathfrak{I} ) }$ is zero.

\item[(2)]  If $K_{0} ( \mathfrak{A} / \mathfrak{I} )_{+} = K_{0} ( \mathfrak{A} / \mathfrak{I} )$, then 
$K_{0} ( \mathfrak{A} )_{+} = K_{0} ( \mathfrak{A} )$.
\item[(3)] If $\mathfrak{A}$ is a unital $C^{*}$-algebra, then 
\begin{itemize}
\item[(a)] $K_{0} ( \mathfrak{I} )$ is finitely generated,

\item[(b)] $\mathrm{rank} ( K_{1} ( \mathfrak{I} ) ) \leq \mathrm{rank} ( K_{0} ( \mathfrak{I} ) )$, and 

\item[(c)] $K_{0}( \mathfrak{I} )_{+} \neq K_{0} ( \mathfrak{I} )$ implies that $ K_{0} ( \mathfrak{I} )\cong \Z$.
\end{itemize}
\end{itemize}  
\end{theor}

\begin{proof}\mbox{}\\

\step{1}{Necessity}
Suppose $\mathfrak{A}$ is a graph $C^{*}$-algebra, i.e. $\mathfrak{A} \cong C^{*} ( G )$ for some graph $G$.  Since $\mathfrak{A}$ has finitely many ideals, $G$ satisfies Condition (K) (see the proof of \cite[Lemma~3.1]{semt:cnga}).  Hence, by \cite[Theorem~2.6]{jhhws:pickdg} (also see \cite[Theorem~3.5]{jaj:rrcag}), $C^{*} ( G )$ has real rank zero and hence, $\mathfrak{A}$ has real rank zero. By \cite[Theorem~3.19]{lgbgkp:crrz}, (1) must then hold.

Suppose $K_{0} ( \mathfrak{A} / \mathfrak{I} )_{+} = K_{0} ( \mathfrak{A} / \mathfrak{I} )$, so that we have that $\mathfrak{A} / \mathfrak{I}$ is a purely infinite simple $C^{*}$-algebra by the dichotomy of simple graph $C^*$-algebras. 
By \cite[Proposition~6.4]{semt:cnga}, $\mathfrak{I}$ is stable. We further prove that $\mathfrak e$ is full. Indeed, since $\A/\I$ is simple, it is enough to show that for some $a\in \A/\I$, $\tau_{\mathfrak e}(a)$ is full in $\corona{\I}$. By \cite[Proposition~3.10]{semt:cnga}, there is a projection $p\in \A$ such that $p\I p$ is stable. By \cite[Theorem~4.23]{lgb:smc}, $p\sim 1_{M(\I)}$. Thus, $\tau_{\mathfrak e}(\pi(p))\sim 1_{\corona{\I}}$, and hence $\tau_{\mathfrak e}(\pi(p))$ is full in $\corona{\I}$. By \cite[Proposition~4.2]{segrer:okfe}, $K_{0} ( \mathfrak{A} )_{+ } = K_{0} ( \mathfrak{A} )$.    

Finally, suppose further that $\mathfrak{A}$ is a unital $C^{*}$-algebra.  Then $K_*(\A)$ is the cokernel and kernel, respectively, of a map from $\Z^{m_1+m_2}$ to  $\Z^{n_1+n_2}$ given by a block triangular matrix in which the $m_1\times n_1$-block specifies $K_*(\I)$ and the $m_2\times n_2$-block specifies $K_*(\A/\I)$. We have 
 $m_i\leq n_i<\infty$, so all these groups are finitely generated, and $\operatorname{rank}K_1(\I)\leq \operatorname{rank}K_0(\I)$,  $\operatorname{rank}K_1(\A)\leq \operatorname{rank}K_0(\A)$ and $\operatorname{rank}K_1(\A/\I)\leq \operatorname{rank}K_0(\A/\I)$, establishing  (3)(a) and (3)(b) in particular.
   
For (3)(c),   suppose $K_{0}( \mathfrak{I} )_{+} \neq K_{0} ( \mathfrak{I} )$.  Then $\mathfrak{I}$ is an AF algebra.  Let $H$ be a nontrivial saturated, hereditary subset of $G^{0}$ corresponding to $\mathfrak{I}$.  Since $\mathfrak{A}$ is unital, $G^{0}$ is a finite set, and hence so is  $H$.  Note that $\mathfrak{I} \otimes \K \cong C^{*} ( E_{H} ) \otimes \K$ with $E_{H} = ( H, r_{G}^{-1} ( H ) , r_{G} \vert_{ r_{G}^{-1} ( H ) }, s_{G} \vert_{ r_{G}^{-1} ( H ) } )$.  Since $E_{H}^{0} = H$ is a finite set and $C^{*} ( E_{H} )$ is a simple AF algebra, $C^{*} ( E_{H} ) \otimes \K \cong \K$  by \cite[Proposition~1.18]{ir:ga}. Therefore, $\mathfrak{I} \otimes \K \cong \K$ and hence $K_0(\I)\cong \Z$.\\[1mm]

We now establish sufficiency.  Thus from here onwards, we assume that (1)--(3) hold.

\medskip

\step{2}{The $\upos{\oo}{*}$ case}
To complete the result in the $\upos{\oo}{*}$ cases, we recall that by \cite[Theorem~5.9]{setkermt:iagc}, it suffices to prove  that every unital quotient of $\mathfrak{A}$ is a Type I $C^{*}$-algebra. This is vacuously true in the $\upos{\oo}{0}$ subcase.

In the $\upos{\oo}{1}$ subcase we note that  since $\mathfrak{A} / \mathfrak{I}$ is a unital, simple graph AF
algebra, by \cite[Proposition~1.18]{ir:ga} again, $\mathfrak{A} / \mathfrak{I} \cong \mathsf{M}_{k}$ which is a  Type I $C^{*}$-algebra.

In the $\upos{\oo}{2}$ subcase, we start by noting that by (3)(c), $K_{0} ( \mathfrak{I} ) = \Z$ since $K_{0} ( \mathfrak{I} )_{+} \neq K_{0} ( \mathfrak{I} )$.  Since  (up to isomorphism) the only simple, nonunital AF algebra with $K_{0} ( \mathfrak{I} ) \cong\Z$ is $\K$, we have that $\mathfrak{I} \cong \K$.  Further, as above we have $\mathfrak{A} / \mathfrak{I} \cong \mathsf{M}_{k}$.  Since $\K$ and $\mathsf{M}_{k}$ are Type I $C^{*}$-algebras, by permanence of this class (cf. \cite[6.2.6]{gkp:cag}), $\mathfrak{A}$ is a Type I $C^{*}$-algebra.  \\[1mm]

\step{3}{Stability and fullness}
We first note that in all remaining cases, $\mathfrak{I}$ is stable. This follows from \cite[Theorem~1.2(i)]{sz:ccrrzcma}  whenever $\I$ is purely infinite, so we may assume that $\I$ is stably finite and appeal to Lemma \ref{newlem}.
We further see that in all remaining cases, $\mathfrak e$ is full. This is again easy to see whenever $\I$ is purely infinite since $M(\I)/\I$ is simple by \cite[Theorem~2.3]{hl:}, and follows from Lemma \ref{newlem} in all other cases.\\[1mm]

\step{4}{The $\upos{\ii}{*}$ and $\upos{\io}{*}$ cases}
Consider first the unital cases  $\upos{\ii}{2}$ and $\upos{\io}{2}$. We know that $\A/\I$ is a unital graph $C^*$-algebra, so as in \mystep{1} we conclude that $K_0(\A/\I)$ is finitely generated and that $\mathrm{rank} ( K_{1} ( \A/\mathfrak{I} ) ) \leq \mathrm{rank} ( K_{0} ( \A/\mathfrak{I} ) )$. This shows that  \cite[Theorem~6.4]{setkmtjw:rking} applies, so there exists a graph $E$ with finitely many vertices such that $C^{*} (E)$ has exactly one proper nontrivial ideal $\mathfrak{I}_{1}$ and $(  \alpha_{\bullet},\beta_{\bullet} ) : \ksix ( \mathfrak{A} , \mathfrak{I} ) \overset{\cong}{\longrightarrow} \ksix ( C^{*} ( E ) , \mathfrak{I}_{1} )$ such that all $\alpha_{\bullet}$ are positive isomorphisms and $\alpha_{2}( [ 1_{ \mathfrak{A} } ] ) = [ 1_{ C^{*} (E) }]$.  Appealing to \cite[Corollary~12]{segr:rccconi} (see \cite[Theorem~2.4]{grer:rccconiII}) in the $\upos{\ii}{2}$ case and to  \cite[Corollary~4.16]{segrer:scecc} in the $\upos{\io}{2}$ case, $\mathfrak{A}\cong C^*(E)$.

In the four remaining cases, we first replace, if necessary, the graph presentation of $\mathfrak I$ by  a left adhesive graph (\cite[Lemma~5.4 ($\ell1$)]{setkmtjw:rking}). By \cite[Proposition~5.5]{setkmtjw:rking} we then obtain a graph $E$ such that $C^{*} (E)$ has exactly one proper nontrivial ideal $\mathfrak{I}_{1}$ and such that  $( \alpha_{\bullet},\beta_\bullet ) :\ksix ( \mathfrak{A} , \mathfrak{I} )  \overset{\cong}{\longrightarrow} \ksix ( C^{*} ( E ) , \mathfrak{I}_{1} )$ can be chosen with all $\alpha_\bullet$  positive isomorphisms and $\alpha_{3}$ scale-preserving (in particular, preserving the class of the unit of the quotient in the $\upos{\ii}{1}$ and  $\upos{\io}{1}$ cases). Appealing to an appropriate classification result we get that $\mathfrak{A} \cong C^{*} (E)$: \cite{mr:ceccstesk} for   $\upos{\ii}{0}$,  \cite[Theorem~2.2]{grer:rccconiII} for $\upos{\ii}{1}$, 
\cite[Theorem~2.3]{segrer:cecc} and \cite[Proposition~2]{segrer:cccabfis} for $\upos{\io}{1}$ and   $\upos{\io}{0}$.\\[1mm]

\step{5}{The $\upos{\oi}{0}$ and $\upos{\oi}{2}$ cases}
In the unital case  $\upos{\oi}{2}$, we note as in \mystep{4} that  $K_0(\A/\I)$ is finitely generated and that $\mathrm{rank} ( K_{1} ( \A/\mathfrak{I} ) ) \leq \mathrm{rank} ( K_{0} ( \A/\mathfrak{I} ) )$. By (3)(c), $K_{0} ( \mathfrak{I} ) \cong \Z$.  Since (up to isomorphism)  the only nonunital simple AF
 algebra with $K_{0}$-group isomorphic to $\Z$ is $\K$, we have that $\mathfrak{I} \cong \K$.  By \cite[Theorem~6.4]{setkmtjw:rking}, there exists a graph $E$ with finitely many vertices such that $C^{*} (E)$ has exactly one proper nontrivial ideal $\mathfrak{I}_{1}$ and there exists an isomorphism $( \alpha_{\bullet} ,\beta_\bullet) : \ksix ( \mathfrak{A} , \mathfrak{I} )  \overset{\cong}{\longrightarrow} \ksix ( C^{*} ( E ) , \mathfrak{I}_{1} )$ such that $\alpha_{\bullet}$ are positive isomorphisms and $\alpha_{2} ( [ 1_{ \mathfrak{A} } ] ) = [ 1_{ C^{*} ( E ) } ] $.  Hence, by \cite[Corollary~4.20]{segrer:scecc}, $\mathfrak{A} \cong C^{*} (E)$.

In the $\upos{\oi}{0}$ case we argue as follows.  By replacing, if necessary, the graph presentation of $\mathfrak A/\mathfrak I$ by  a right adhesive graph (\cite[Lemma~5.4 (r1)]{setkmtjw:rking}) we create by \cite[Proposition~5.5]{setkmtjw:rking} a graph $\A$ such that $C^{*} (E)$  has exactly one proper nontrivial ideal $\mathfrak{I}_{1}$ and there exists an isomorphism $( \alpha_\bullet,\beta_{\bullet} ) : \ksix ( \mathfrak{A} , \mathfrak{I} )  \overset{\cong}{\longrightarrow} \ksix ( C^{*} ( E ) , \mathfrak{I}_{1} )$ such that all $\alpha_\bullet$ are positive isomorphisms.   By   \cite[Theorem~2.3]{segrer:cecc} and \cite[Proposition~2]{segrer:cccabfis}, $\mathfrak{A} \cong C^{*} (E)$. \\[1mm]

\step{6}{The $\upos{\io}{1}$ case}
Because of Theorem \ref{birthdaypresent}, we only need to check that   the invariant in \eqref{Ksix} satisfies the conditions in Theorem \ref{takeshitheor}. We see that (vi) holds  by condition (2), and we get (vii)--(ix) because we are in the $\upos{\io}{1}$ case.

Condition (i) follows from the fact that $\I$ is a simple, stable  AF algebra. The first half of (ii) --- that the extension is lexicographic in the sense of Handelman (see \cite{dh:eacdg}) --- follows from Lemma \ref{basicD} and \cite[Corollary~3.22]{segrer:okfe} as a direct consequence of the fullness of $\mathfrak e$,  which we established in Step~3. The second half follows by noting that any lift of  $1_{\A/\I}$ lies in the intersection. Condition (iii) follows because $\mathfrak D$ is a nonunital AF algebra, and (iv) and (v) follow by construction.
\end{proof}

\begin{corol}\label{bentmannplus}
Let $C^*(E_1)$ and $C^*(E_3)$ be unital and simple graph $C^*$-algebras and consider the unital extension
\[
\xymatrix{0\ar[r]&C^*(E_1)\otimes \K\ar[r]&\mathfrak{X}\ar[r]&C^*(E_3)\ar[r]&0.}
\]
The following are equivalent
\begin{enumerate}[(a)]
\item $\mathfrak X$ is a graph $C^*$-algebra;
\item $\mathfrak X$ has real rank zero; 
\item $\partial_0:K_0(C^*(E_3))\to K_1(C^*(E_1))$ vanishes.
 \end{enumerate}
\end{corol}
\begin{proof}
As noted in Step~1 of the proof of Theorem \ref{thm:extensions}, we know that $C^*(E_1)$ and $C^*(E_3)$ have real rank zero from the outset. Hence, by \cite[Theorem~3.19]{lgbgkp:crrz}, (b) and (c) are equivalent.  Since $C^*(E_3)$ is simple, the Busby invariant of the above extension is either the zero map or injective.  In the latter case, we have that every nonzero ideal of $\mathfrak X$ has a nontrivial intersection with $C^*(E_1) \otimes \K$.  Thus, as $C^*(E_3)$ and $C^*(E_1)$ are assumed to be simple, either $\mathfrak X$ is isomorphic to $C^*(E_1)\otimes \K\oplus C^*(E_3)$ (the Busby map is the zero map)  or $C^*(E_1)\otimes \K$ is the only proper nontrivial ideal of $\mathfrak X$ (the Busby map is injective).  We conclude that $C^*(E_1)\otimes \K$ is the only proper nontrivial ideal of $\mathfrak X$ since $\mathfrak X$ is unital and $C^*(E_1)\otimes \K\oplus C^*(E_3)$ is nonunital.  Consequently we can apply Theorem \ref{thm:extensions} since $C^*(E_1)\otimes \K$ is itself a graph $C^*$-algebra. When (a) holds, by Theorem~\ref{thm:extensions} we get (c). In the other direction, we need to establish (2) and (3) of Theorem \ref{thm:extensions} separately under the assumption of (c).

Here, (3) follows from the fact that $C^*(E_1)$ is unital and that any unital graph $C^*$-algebra satisfies (a)--(c) of (3) in Theorem~\ref{thm:extensions}, so it remains to establish (2). There is nothing to check in the $\upos{\oo}{2}$ and  $\upos{\io}{2}$ cases.  For the cases, $\upos{\ii}{2}$ and $\upos{\oi}{2}$, by  \cite[Corollary~3.22]{segrer:okfe} it is enough to show that the extension is full.  In both cases, $\corona{ C^*(E_1) \otimes \K}$ is simple since $C^*(E_1) \otimes \K$ is either a purely infinite simple $C^*$-algebra or isomorphic to $\K$.  We conclude that the extension is full in the $\upos{\ii}{2}$ and $\upos{\oi}{2}$ cases since the Busby invariant of the extension is injective. 
\end{proof}

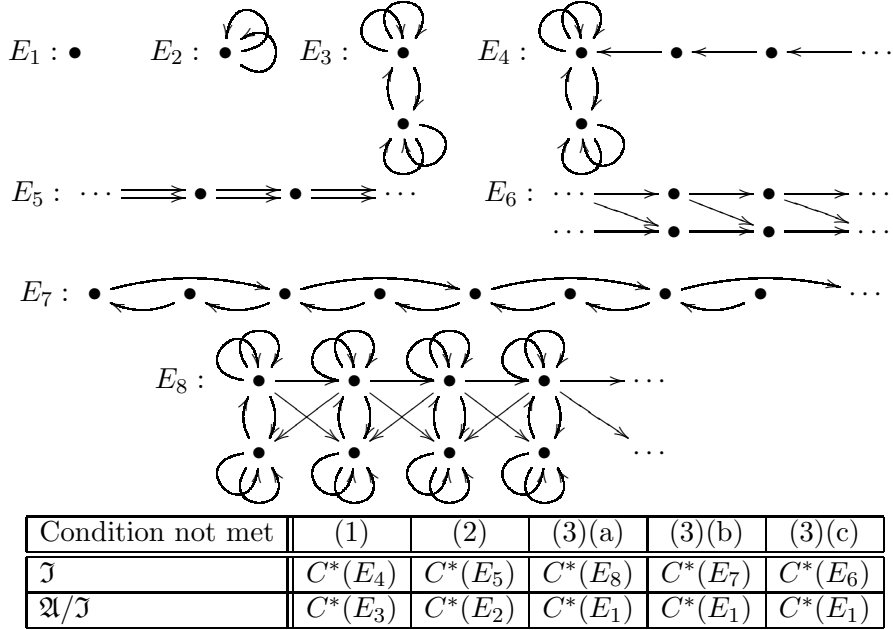
\begin{figure}[h]
\begin{center}
$E_1:$ $\bullet$ \qquad  $E_2:$ $\xymatrix{\bullet\ar@(r,u)[]\ar@(dr,ur)[]}$\qquad
$E_3:\,\,\,\,\,\,$ $\xymatrix@R=0.55cm{\bullet\ar@/^/[d]\ar@(l,u)[]\ar@(ul,ur)[]\\\bullet\ar@/^/[u]\ar@(r,d)[]\ar@(dr,dl)[]}$\qquad
$E_4:\,\,\,\,\,\,$ $\xymatrix@R=0.55cm{\bullet\ar@/^/[d]\ar@(l,u)[]\ar@(ul,ur)[]&\bullet\ar[l]&\bullet\ar[l]&\cdots\ar[l]\\\bullet\ar@/^/[u]\ar@(r,d)[]\ar@(dr,dl)[]}$
\\[5mm]
$E_5:$ $\xymatrix{\cdots\ar@<-0.5mm>[r]\ar@<0.5mm>[r]&\bullet\ar@<-0.5mm>[r]\ar@<0.5mm>[r]&\bullet\ar@<-0.5mm>[r]\ar@<0.5mm>[r]&\cdots}$\qquad
$E_6:$ $\xymatrix@R=0.1cm{\cdots\ar[r]\ar[dr]&\bullet\ar[r]\ar[dr]&\bullet\ar[r]\ar[dr]&\cdots\\
\cdots\ar[r]&\bullet\ar[r]&\bullet\ar[r]&\cdots}$\\[4mm]
$E_7:$ $\xymatrix{\bullet\ar@/^/[rr]&\bullet\ar@/^/[l]&\bullet\ar@/^/[l]\ar@/^/[rr]&\bullet\ar@/^/[l]&\bullet\ar@/^/[l]\ar@/^/[rr]&\bullet\ar@/^/[l]&\bullet\ar@/^/[l]\ar@/^/[rr]&\bullet\ar@/^/[l]&\cdots}$\\[7mm]
$E_8:\quad$ $\xymatrix@R=0.55cm{\bullet\ar@(l,u)\ar@(ul,ur)\ar@/^/[d]\ar[dr]\ar[r]&\bullet\ar[r]\ar@(l,u)\ar@(ul,ur)\ar@/^/[d]\ar[dr]\ar[dl]&\bullet\ar[r]\ar@(l,u)\ar@(ul,ur)\ar@/^/[d]\ar[dr]\ar[dl]&\bullet\ar[r]\ar@(l,u)\ar@(ul,ur)\ar[dl]\ar[dr]\ar@/^/[d]&\cdots\\\bullet\ar@(l,d)\ar@(dl,dr)\ar@/^/[u]&\bullet\ar@(l,d)\ar@(dl,dr)\ar@/^/[u]&\bullet\ar@(l,d)\ar@(dl,dr)\ar@/^/[u]&\bullet\ar@(l,d)\ar@(dl,dr)\ar@/^/[u]&\cdots&}$\\[6mm]
\begin{tabular}{|l||c|c|c|c|c|}\hline
Condition not met&(1)&(2)&(3)(a)&(3)(b)&(3)(c)\\\hline\hline
$\I$&$C^*(E_4)$&$C^*(E_5)$&$C^*(E_8)$&$C^*(E_7)$&$C^*(E_6)$\\\hline
$\A/\I$&$C^*(E_3)$&$C^*(E_2)$&$C^*(E_1)$&$C^*(E_1)$&$C^*(E_1)$\\\hline
\end{tabular}
\end{center}
\caption{Examples showing necessity of the conditions in Theorem \ref{thm:extensions}}\label{graphexx}
\end{figure}

\begin{examp}
The graphs given in Figure \ref{graphexx} can be used to demonstrate necessity of all the individual  conditions in Theorem \ref{thm:extensions} in the sense of producing an extension $\mathfrak e$ satisfying all conditions of Theorem \ref{thm:extensions} but one, and where $\A$ is not a graph $C^*$-algebra. We have that $C^*(E_1)=\mathbb C$, $C^*(E_2)=\mathcal O_2$, $C^*(E_5)={\mathsf M}_{2^\infty}\otimes\K$, that $C^*(E_4)$, $C^*(E_7)$ and $C^*(E_8)$ are the stable UCT Kirchberg algebras with $K$-groups $\Z\oplus\Z$, $0\oplus \Z$ and $\Z^\infty\oplus 0$, respectively, and that $C^*(E_3)$ is the unital UCT Kirchberg algebra with $K$-groups $\Z\oplus\Z$ and vanishing class of the identity.
Finally,  $C^*(E_6)$ is the stable AF algebra with dimension group $\Z+\varphi\Z$ where $\varphi$ is the golden mean.

In all cases but the one with condition (2) this amounts to arranging the $K$-theory in an obvious way clashing with the given condition, finding first an appropriate $K\!K_*$-element and realizing it as explained in \cite{mr:ceccstesk}. The  example to show necessity of (2) is more involved and was explained in \cite[Example~4.1]{segrer:okfe}.
\end{examp}

\begin{remar}
Note that although $C^*(E)\oplus C^*(F)$ is isomorphic to  the graph $C^*$-algebra $C^*(E\sqcup F)$, condition (2) of Theorem \ref{thm:extensions} is not automatically met in this case. It is possible to obtain a result which works for any  situation when $\A$ has an ideal $\I$ so that $\I$ and $\A/\I$ are both simple graph $C^*$-algebras, 
%by replacing (2) by
%\begin{itemize}
%\item[(2')] If $K_{0} ( \mathfrak{A} / \mathfrak{I} )_{+} = K_{0} ( \mathfrak{A} / \mathfrak{I} )$, then 
%$0=[e]$ in  $K_{0} ( \mathfrak{A} )$ for some projection $e\in \A\backslash \I$.
%\end{itemize}
but since this is somewhat convoluted, we will refrain from doing so here.
\end{remar}

\subsection*{Acknowledgments}

 The first named author was supported by the DFF-Research Project 2 `Automorphisms and Invariants of Operator Algebras', no. 7014-00145B, and by the Danish National Research Foundation through the Centre for Symmetry and Deformation (DNRF92). The second named author was funded by the Carlsberg Foundation through an Internationalisation Fellowship. The fourth and fifth named authors were supported by Simons Foundation Collaboration Grants, \#567380 to Ruiz and \#527708 to Tomforde.  We thank the anonymous referee for a careful reading of an earlier version of this paper, which improved the exposition and led to the correction of several potentially confusing mistakes.

\providecommand{\bysame}{\leavevmode\hbox to3em{\hrulefill}\thinspace}
\providecommand{\MR}{\relax\ifhmode\unskip\space\fi MR }
\providecommand{\MRhref}[2]{%
  \href{http://www.ams.org/mathscinet-getitem?mr=#1}{#2}
}
\providecommand{\href}[2]{#2}


\begin{thebibliography}{EKTW16}

\bibitem[Ark13]{sea:dpcke}
S.E. Arklint, \emph{Do phantom {C}untz-{K}rieger algebras exist?}, Operator
  algebra and dynamics, Springer Proc. Math. Stat., vol.~58, Springer,
  Heidelberg, 2013, pp.~31--40. %\MR{3142030}

\bibitem[Ben19]{rb:ecka}
R.~Bentmann, \emph{Extensions of {C}untz-{K}rieger algebras}, J. Math. Anal. Appl. \textbf{471} (2019), 647--652.

\bibitem[BO08]{npbno:cfa}
N.~P. Brown and N.~Ozawa, \emph{{$C^*$}-algebras and finite-dimensional
  approximations}, Graduate Studies in Mathematics, vol.~88, American
  Mathematical Society, Providence, RI, 2008. %\MR{2391387}

\bibitem[BP91]{lgbgkp:crrz}
L.G. Brown and G.K. Pedersen, \emph{{$C^*$}-algebras of real rank zero}, J.\
  Funct.\ Anal. \textbf{99} (1991), 131--149.

\bibitem[Bro88]{lgb:smc}
L.G. Brown, \emph{Semicontinuity and multipliers of {$C^*$}-algebras}, Canad.\
  J.\ Math. \textbf{XL} (1988), no.~4, 865--988.

\bibitem[Cun81]{jc:cctmc2}
J.~Cuntz, \emph{A class of {$C^*$}-algebras and topological markov chains {II}:
  Reducible chains and the {E}xt-functor for {$C^*$}-algebras}, Invent.\ Math.
  \textbf{63} (1981), 25--40.

\bibitem[DL94]{mdtal:ecrrzc}
M.~{D\u{a}d\u{a}rlat} and T.A. Loring, \emph{Extensions of certain real rank
  zero ${C}^*$-algebras}, Ann.\ Inst.\ Fourier \textbf{44} (1994), 907--925.

\bibitem[Eff81]{ege:dca}
E.G. Effros, \emph{Dimensions and {$C^*$}-algebras}, CBMS Regional Conf.\ Ser.\
  in Math., no.~46, American Mathematical Society, 1981.

\bibitem[EHS80]{ehs}
E.G. Effros, D.E. Handelman, and C.L. Shen, \emph{Dimension groups and their
  affine representations}, Amer. J. Math. \textbf{102} (1980), no.~2, 385--407.
  %\MR{83g:46061}

\bibitem[EK01]{gaedk:avbfat}
G.A. Elliott and D.~Kucerovsky, \emph{An abstract
  {V}oiculescu-{B}rown-{D}ouglas-{F}illmore absorption theorem}, Pacific J.
  Math. \textbf{198} (2001), no.~2, 385--409. %\MR{MR1835515 (2002i:46052)}

\bibitem[EKRT14]{setkermt:iagc}
S.~Eilers, T.~Katsura, E.~Ruiz, and M.~Tomforde, \emph{Identifying
  {AF}-algebras that are graph {$C^*$}-algebras}, J. Funct. Anal. \textbf{266}
  (2014), 3968--3996.

\bibitem[EKTW16]{setkmtjw:rking}
S.~Eilers, T.~Katsura, M.~Tomforde, and J.~West, \emph{The ranges of
  {$K$}-theoretical invariants for nonsimple graph algebras}, Trans. Amer.
  Math. Soc. \textbf{368} (2016), 3811--3847.

\bibitem[Ell76]{gae:cilssfa}
G.A. Elliott, \emph{On the classification of inductive limits of sequences of
  semisimple finite-dimensional algebras}, J. Algebra \textbf{38} (1976),
  no.~1, 29--44.

\bibitem[ER06]{segr:rccconi}
S.~Eilers and G.~Restorff, \emph{On {R\o rdam's} classification of certain
  {$C^*$}-algebras with one nontrivial ideal}, Operator algebras: The Abel
  symposium 2004, Abel Symposia, no.~1, Springer-Verlag, 2006, pp.~87--96.

\bibitem[ERR]{segrer:scecc}
S.~Eilers, G.~Restorff, and E.~Ruiz, \emph{Strong classification of extensions
  of classifiable {$C^*$}-algebras}, preprint, arXiv:1301.7695.

\bibitem[ERR09]{segrer:cecc}
\bysame, \emph{Classification of extensions of classifiable {$C^*$}-algebras},
  Adv. Math. \textbf{222} (2009), 2153--2172.

\bibitem[ERR14]{segrer:okfe}
\bysame, \emph{The ordered {$K$}-theory of a full extension}, Canad. J. Math
  \textbf{66} (2014), 596--624.

\bibitem[ERR16]{segrer:cccabfis}
\bysame, \emph{Corrigendum to ``{C}lassifying {$C^*$}-algebras with both finite
  and infinite subquotients" [{J. Funct. Anal.} 265 (2013) 449-468]},
  J. Funct. Anal. \textbf{270} (2016), 854--859.

\bibitem[ET10]{semt:cnga}
S.~Eilers and M.~Tomforde, \emph{On the classification of nonsimple graph
  algebras}, Math. Ann. \textbf{346} (2010), 393--418.

\bibitem[FLR00]{flr:graph} N.J.~Fowler, M.~Laca, and I.~Raeburn, \emph{The {$C^*$}-algebras of infinite graphs}, Proc. Amer. Math. Soc. \textbf{128} (2000), no. 8, 2319--2327. %,  \texttt{doi:10.1090/S0002-9939-99-05378-2}.

\bibitem[Gab16]{jg:nnae}
J.~Gabe, \emph{A note on nonunital absorbing extensions}, Pacific J. Math.
  \textbf{284} (2016), no.~2, 383--393. %\MR{3544306}

\bibitem[GR18]{jger:}
J.~Gabe and E.~Ruiz, \emph{The unital Ext-groups and classification of $C^*$-algebras}, to appear in  Glasgow Math. J.,  \texttt{doi:10.1017/S0017089519000053}.

\bibitem[Han82]{dh:eacdg}
D.~Handelman, \emph{Extensions for {AF} {$C^{\ast}$}\ algebras and dimension
  groups}, Trans. Amer. Math. Soc. \textbf{271} (1982), no.~2, 537--573.
  %\MR{654850}

\bibitem[HR98]{jhmr:sc}
J.~Hjelmborg and M.~R\o{}rdam, \emph{On stability of {$C^*$}-algebras}, J.
  Funct. Anal. \textbf{155} (1998), 153--170.

\bibitem[HS03]{jhhws:pickdg}
J.H. Hong and W.~Szyma{\'n}ski, \emph{Purely infinite {C}untz-{K}rieger
  algebras of directed graphs}, Bull. London Math. Soc. \textbf{35} (2003),
  no.~5, 689--696. %\MR{MR1989499 (2005c:46097)}

\bibitem[Jeo04]{jaj:rrcag}
J.A Jeong, \emph{Real rank of {$C^*$}-algebras associated with graphs}, J.
  Aust. Math. Soc. \textbf{77} (2004), no.~1, 141--147. %\MR{2069031
  %(2005b:46118)}

\bibitem[Kir00]{ek:nkmkna}
E.~Kirchberg, \emph{Das nicht-kommutative Michael-Auswahlprinzip und die Klassifikation nicht-einfacher Algebren}. 
$C\sp *$-algebras (M\"unster, 1999), pp. 92-141. Springer, Berlin 2000.


\bibitem[KR00]{ekmr:npic}
E.~Kirchberg and M.~R{\o}rdam, \emph{Non-simple purely infinite
  ${C}^*$-algebras}, Amer. J. Math. \textbf{122} (2000), 637--666.

\bibitem[KST09]{tkasmk:ragaelua}
T.~Katsura, A.~Sims, and M.~Tomforde, \emph{Realization of {AF}-algebras as
  graph algebras, {E}xel-{L}aca algebras, and ultragraph algebras}, J. Funct.
  Anal. \textbf{257} (2009), no.~5, 1589--1620. %\MR{2541281 (2010g:46103)}

\bibitem[Lin89]{hl:}
H.~Lin, \emph{The simplicity of the quotient algebra $M(A)/A$
 of a simple $C^*$-algebra}, Math. Scand. \textbf{65} (1989), 119-128. 

\bibitem[LR95]{hlmr:eilca}
H.~Lin and M.~R\o{}rdam, \emph{Extensions of inductive limits of circle
  algebras}, J.\ London Math.\ Soc. \textbf{51} (1995), no.~2, 603--613.

\bibitem[Ped79]{gkp:cag}
G.K. Pedersen, \emph{{$C^*$}-algebras and their automorphism groups}, Academic
  Press, London, 1979.

\bibitem[Rae05]{ir:ga}
I.~Raeburn, \emph{Graph algebras}, CBMS Regional Conference Series in
  Mathematics, vol. 103, Published for the Conference Board of the Mathematical
  Sciences, Washington, DC, 2005. %\MR{MR2135030 (2005k:46141)}

\bibitem[Res06]{gr:cckasi}
G.~Restorff, \emph{Classification of {Cuntz-Krieger} algebras up to stable
  isomorphism}, J. Reine Angew. Math. \textbf{598} (2006), 185--210.

\bibitem[R{\o}r97]{mr:ceccstesk}
M.~R{\o}rdam, \emph{Classification of extensions of certain ${C}^*$-algebras by
  their six term exact sequences in ${K}$-theory}, Math. Ann. \textbf{308}
  (1997), no.~1, 93--117.

\bibitem[RR07]{grer:rccconiII}
G.~Restorff and E.~Ruiz, \emph{On {R\o rdam's} classification of certain
  {$C^*$}-algebras with one nontrivial ideal {II}}, {Math. Scand.} \textbf{101}
  (2007), 280--292.

\bibitem[Spi88]{jss:eceia}
J.S. Spielberg, \emph{Embedding {$C^*$}-algebra extensions into {$AF$}
  algebras}, J. Funct. Anal. \textbf{81} (1988), no.~2, 325--344. %\MR{MR971884
  %(90a:46154)}

\bibitem[Zha92]{sz:ccrrzcma}
S.~Zhang, \emph{Certain {$C^\ast$}-algebras with real rank zero and their
  corona and multiplier algebras. {I}}, Pacific J. Math. \textbf{155} (1992),
  no.~1, 169--197. %\MR{1174483}

\end{thebibliography}
\end{document}